\documentclass[12pt,english]{article}
\usepackage{babel}
\usepackage{units}
\usepackage{textcomp}
\usepackage{mathrsfs}
\usepackage{amsmath}
\usepackage{amsthm}
\usepackage{amssymb}
\usepackage[unicode=true,
 bookmarks=true,bookmarksnumbered=false,bookmarksopen=false,
 breaklinks=true,pdfborder={0 0 0},pdfborderstyle={},backref=false,colorlinks=true]
 {hyperref}
\hypersetup{pdftitle={Singularity of time-change processes},
 pdfauthor={Jos{\'e} Lu{\'\i}s da Silva and Mohamed Erraoui},
 pdfsubject={Singularity of measures on path spaces},
 pdfkeywords={Grey Brownian motion, p-variation, Subordination},
 linkcolor=black,citecolor=blue,urlcolor=blue,filecolor=blue}

\makeatletter

\theoremstyle{plain}
\newtheorem{thm}{\protect\theoremname}
\theoremstyle{remark}
\newtheorem{rem}[thm]{\protect\remarkname}
\theoremstyle{plain}
\newtheorem{lem}[thm]{\protect\lemmaname}
\theoremstyle{plain}
\newtheorem{prop}[thm]{\protect\propositionname}
\theoremstyle{definition}
\newtheorem{defn}[thm]{\protect\definitionname}
\theoremstyle{definition}
\newtheorem{example}[thm]{\protect\examplename}

\makeatother

\providecommand{\definitionname}{Definition}
\providecommand{\examplename}{Example}
\providecommand{\lemmaname}{Lemma}
\providecommand{\propositionname}{Proposition}
\providecommand{\remarkname}{Remark}
\providecommand{\theoremname}{Theorem}

\begin{document}
\title{Singularity of Generalized Grey Brownian Motion and Time-Changed Brownian
Motion}
\author{\textbf{Jos{\'e} Lu{\'\i}s da Silva},\\
 CIMA, University of Madeira, Campus da Penteada,\\
 9020-105 Funchal, Portugal.\\
 Email: joses@staff.uma.pt\and \textbf{Mohamed Erraoui}\\
 Universit{\'e} Cadi Ayyad, Facult{\'e} des Sciences Semlalia,\\
 D{\'e}partement de Math{\'e}matiques, BP 2390, Marrakech, Maroc\\
 Email: erraoui@uca.ma}
\maketitle
\begin{abstract}
The generalized grey Brownian motion is a time continuous self-similar
with stationary increments stochastic process whose one dimensional
distributions are the fundamental solutions of a stretched time fractional
differential equation. Moreover, the distribution of the time-changed
Brownian motion by an inverse stable process solves the same equation,
hence both processes have the same one dimensional distribution. In
this paper we show the mutual singularity of the probability measures
on the path space which are induced by generalized grey Brownian motion
and the time-changed Brownian motion though they have the same one
dimensional distribution. This singularity property propagates to
the probability measures of the processes which are solutions to the
stochastic differential equations driven by these processes.\medskip{}

\end{abstract}

\section{Introduction }

Over the past decades, the physical and mathematical community has
shown a growing interest in modeling anomalous diffusion processes.
The label \emph{anomalous diffusion} is assigned to processes whose
particle displacement variance does not grow linearly in time, in
opposition to the standard Brownian motion that is mainly characterized
by a linear law. In this respect, the terms subdiffusion and superdiffusion
are used for those processes whose variance grows in time is slower
or faster than linear, respectively.

The generalized grey Brownian motion (ggBm for short) $B_{\beta,\alpha}$,
$0<\beta\leq1$, $0<\alpha<2$ was introduced by \cite{Mura_Pagnini_08}
to model both subdiffusion and superdiffusion, see also \cite{Mura_mainardi_09}.
This family of stochastic processes are $\frac{\alpha}{2}$-self-similar
with stationary increments and includes in particular the class of
fractional Brownian motion (fBm) $B^{H}=B_{1,\frac{\alpha}{2}}$ with
Hurst parameter $H=\frac{\alpha}{2}$ and $\beta=1$ was well as standard
Brownian motion (Bm) $W=B_{1,1}$ with $\alpha=\beta=1$. For $\alpha=\beta\in(0,1)$
we obtain grey Brownian motion (gBm) $B_{\beta}=B_{\beta,\beta}$,
$0<\beta\le1$ introduced by \cite{Schneider90} to
study time-fractional diffusion equation. Moreover, the probability
density function (PDF) $f_{\beta,\alpha}$ of ggBm is the fundamental
solution of the stretched time-fractional standard diffusion equation

\begin{equation}
u(x,t)=u_{0}(x)+\frac{1}{\Gamma(\beta)}\frac{\alpha}{\beta}\int_{0}^{t}s^{\nicefrac{\alpha}{\beta-1}}\left(t^{\nicefrac{\alpha}{\beta}}-s^{\nicefrac{\alpha}{\beta}}\right)^{\beta-1}\frac{\partial^{2}}{\partial x^{2}}u(x,s)\,ds,\quad t\geq0,\label{eq:master}
\end{equation}
with initial condition $u_{0}(x)=\delta(x)$, see \cite{Mainardi_Mura_Pagnini_2010}
and references therein. Therefore, the Fourier transform of $f_{\beta,\alpha}$
is 
\[
\mathbb{E}\big(e^{i\theta B_{\beta,\alpha}(t)}\big)=E_{\beta}\left(-\frac{\theta^{2}}{2}t^{\alpha}\right),\quad\theta\in\mathbb{R},\,t\ge0,
\]
where $E_{\beta}$ is the Mittag-Leffler function, see Section \ref{sec:Realization}
for the definition and properties the this function.

It is important to observe that, starting from a master equation \eqref{eq:master}
which describes the dynamic evolution of a probability density function
$f_{\beta,\alpha}$, it is always possible to define an equivalence
class of stochastic processes with the same PDF function $f_{\beta,\alpha}$.
All these processes provide suitable stochastic models for the starting
equation \eqref{eq:master}. One way to do this is to use a subordination
technique. Let us recall that heuristically, the ggBm $B_{\beta,\alpha}$
cannot be a subordinated process (for example if $\beta=1$ it reduces
to a fractional Brownian motion).

Indeed, let $S_{\beta}$ denotes a strictly increasing $\beta$-stable
subordinator and $U_{\beta}$ its inverse, that is 
\[
U_{\beta}(t):=\inf\left\{ s\ge0\middle|S_{\beta}(t)>s\right\} ,\quad t\in\mathbb{R}^{+}:=[0,\infty).
\]
We consider the increasing, right continuous process given by $
t\mapsto U_{\beta,\alpha}(t):=U_{\beta}(t^{\nicefrac{\alpha}{\beta}}).
$
Now let us introduce the process $X_{\beta,\alpha}$ defined by evaluating
Brownian motion (Bm) $W$, independent of the subordinator $S_{\beta}$
and its inverse $U_{\beta}$, at the random time $U_{\beta,\alpha}$,
this is 
\[
X_{\beta,\alpha}(t):=W(U_{\beta,\alpha}(t)),\quad t\ge0.
\]
We will termed it as continuous time-changed Brownian motion. Since
$U_{\beta}$ has a Mittag-Leffler distribution 
\begin{equation}
\mathbb{E}(e^{-uU_{\beta}(s)})=E_{\beta}(-us^{\beta})\quad,u>0,\;s\in\mathbb{R}^{+}.\label{eq:Laplace-U}
\end{equation}
cf.\ \cite{Bingham1971}, then $X_{\beta,\alpha}$
has Laplace transform given by 
\[
\mathbb{E}\big(e^{-uX_{\beta,\alpha}(t)}\big)=E_{\beta}\left(-\frac{u^{2}}{2}t^{\alpha}\right),\quad u>0,\,t\ge0.
\]
Then it follows that the PDF function of the process $X_{\beta,\alpha}$
is the fundamental solution of the stretched time-fractional equation
\eqref{eq:master}. But this equation describes only the evolution
in time of one-dimensional distributions; thus it is mathematical
incomplete for the identification of the process. Clearly we need
to know all finite dimensional distributions.

Here we would like to mention a closed similar result associated to
the fractional Poisson process (fPp) $N_{\beta}$, $0<\beta\le1$,
see \cite{MGS04,Meerschaert2011,Beghin:2009fi}.
More precisely, the fPp is a natural generalization of the standard
Poisson process $N$ with intensity $\lambda>0$. The corresponding
iid waiting times $J_{n}$ are Mittag-Leffler distributed, that is
\[
P(J_{n}>t)=E_{\beta}(-\lambda t^{\beta}).
\]
Then the fPp 
\[
N_{\beta}(t)=\max\left\{ n\ge0\middle|J_{1}+\ldots+J_{n}\le t\right\} 
\]
is a renewal process with Mittag-Leffler waiting times and its distribution
is given by 
\[
p_{n}(t):=P(N_{\beta}(t)=n)=\frac{(\lambda t^{\beta})^{n}}{n!}E_{\beta}^{(n)}(-\lambda t^{\beta}),\quad n\ge0,\;t>0,
\]
with $E_{\beta}^{(n)}(-\lambda t^{\beta}):=\frac{d^{n}}{dz^{n}}E_{\beta}(z)|_{z=-\lambda t^{\beta}}$.
Moreover the probability generating function $G_{\beta}$ of $N_{\beta}$
is 
\[
G_{\beta}(z,t):=\mathbb{E}\big(z^{N_{\beta}(t)}\big)=E_{\beta}\big(\lambda t^{\beta}(z-1)\big).
\]
Consider now the time-changed process $N(U_{\beta}(t))$, called in
\cite{Meerschaert2011} fractal time Poisson process (ftPp). \cite{Beghin:2009fi}
showed that both processes $N_{\beta}$ and $N(U_{\beta})$ have the
same one-dimensional distributions since they solved the same time-fractional
analogue of the forward Kolmogorov equation with the same point source.
So far, we have a similar scenario as ggBm $B_{\beta,\alpha}$ and
the time-changed process $W(U_{\beta,\alpha})$, that is they have
the same one dimensional density functions. The surprising fact here
is indeed that the fPp $N_{\beta}$ and the time-changed process $N(U_{\beta})$
are the same process. The reason for this is: $N_{\beta}$ and $N(U_{\beta})$
are both pure jump processes with iid Mittag-Leffler distributed waiting
times between jumps, see \cite[Theorem~2.2]{Meerschaert2011}. So
it is natural to ask whether the processes $B_{\beta,\alpha}$ and
$W(U_{\beta,\alpha})$ are same or not. Unfortunately this is not
the case. This difference with the Poissonian case is essentially
a consequence of the fact that $B_{\beta,\alpha}$ and $W(U_{\beta,\alpha})$
have not the same $p$-variation index. Indeed they are not both semimartingales,
specifically $B_{\beta,\alpha}$ is not a semimartingale. Moreover,
their probability laws $P_{B_{\beta,\alpha}}$ and $P_{W(U_{\beta,\alpha})}$,
defined on $C([0,1],\mathbb{R})$ (the space of continuous functions
$w:[0,1]\longrightarrow\mathbb{R}$ such that $w(0)=0$), are mutually
singular, see Theorem\ \ref{thm:Singularity} below.

In Section \ref{sec:Realization} we define the families of processes
used in the paper, their main properties and canonical realizations.
In Section\ \ref{sec:Singularity} we prove the mutual singularity
of the probability measures $P_{B_{\beta,\alpha}}$ and $P_{W(U_{\beta,\alpha})}$.
We will determine explicit events which distinguish between the two
measures. The key tool is the $p$-variation of the paths of the corresponding
processes. In Section\ \ref{sec:Applications-SDEs} we apply the
same idea to show that the singularity property propagates to the
solutions of stochastic differential equations (SDEs) driven by ggBm
and time-changed process $W(U_{\beta,\alpha})$ of the form 
\[
X(t)=X(0)+\int_{0}^{t}g(X(s))\,dB_{\beta,\alpha}(s),\;0\le t\le1,
\]
and 
\[
Y(t)=Y(0)+\int_{0}^{t}f(Y(s))\,dW(U_{\beta,\alpha}(s)),\quad0\le t\le1.
\]

\label{sec:Realization}In this section we introduce the required
path spaces in order to realize the stochastic process we deal with
as canonical processes. These spaces are subsets of $C(\mathbb{T},\mathbb{R})$,
the space continuous functions $w:\mathbb{T}\longrightarrow\mathbb{R}$
such that $w(0)=0$. The time parameter set $\mathbb{T}$ is either
$\mathbb{R}^{+}$ or $[0,1]$ depending on the process in question,
we have in mind namely the ggBm $B_{\beta,\alpha}$ and the subordination
of the Bm $W$ by a family of time-change $U_{\beta,\alpha}$, see
details below. The functions $Y(t)$, $t\in\mathbb{T}$ taking values
in $\mathbb{R}$, defined by $Y(t)(w):=w(t)$ are called \emph{coordinate
mappings}.

\subsection{Canonical Realization of Generalized Grey Brownian Motion}

Let $0<\beta<1$ and $1<\alpha<2$ be given. A continuous stochastic
process defined on a complete probability space $\left(\Omega,\mathcal{F},P\right)$
is a generalized grey Brownian motion, denoted by $B_{\beta,\alpha}=\{B_{\beta,\alpha}(t),\,t\geq0\}$,
see \cite{Mura_mainardi_09}, if: 
\begin{enumerate}
\item $B_{\beta,\alpha}(0)=0$, $P$ almost surely. 
\item Any collection $\big\{ B_{\beta,\alpha}(t_{1}),\ldots,B_{\beta,\alpha}(t_{n})\big\}$
with $0\leq t_{1}<t_{2}<\ldots<t_{n}<\infty$ has characteristic function
given, for any $\theta=(\theta_{1},\ldots,\theta_{n})\in\mathbb{R}^{n}$,
by 
\begin{equation}
\mathbb{E}\left(\exp\left(i\sum_{k=1}^{n}\theta_{k}B_{\beta,\alpha}(t_{k})\right)\right)=E_{\beta}\left(-\frac{1}{2}\theta^{\top}\Sigma_{\alpha}\theta\right),\label{eq:charact-func-ggBm}
\end{equation}
where 
\[
\Sigma_{\alpha}=\big(t_{k}^{\alpha}+t_{j}^{\alpha}-|t_{k}-t_{j}|^{\alpha}\big)_{k,j=1}^{n}
\]
and the joint probability density function is equal to 
\[
f_{\beta}(\theta,\Sigma_{\alpha})=\frac{(2\pi)^{-\frac{n}{2}}}{\sqrt{\det\Sigma_{\alpha}}}\int_{0}^{\infty}\tau^{-\frac{n}{2}}e^{-\frac{1}{2\tau}\theta^{\top}\Sigma_{\alpha}^{-1}\theta}M_{\beta}(\tau)\,d\tau.
\]
\end{enumerate}
Here $E_{\beta}$ is the Mittag-Leffler (entire) function 
\[
E_{\beta}(z)=\sum_{n=0}^{\infty}\frac{z^{n}}{\Gamma(\beta n+1)},\quad z\in\mathbb{C},
\]
and where $M_{\beta}$ is the so-called $M$-Wright probability density
function with Laplace transform 
\begin{equation}
\int_{0}^{\infty}e^{-s\tau}M_{\beta}(\tau)\,d\tau=E_{\beta}(-s).\label{eq:M_wright}
\end{equation}
The density $M_{\beta}$ is a particular case of the Wright function
$W_{\lambda,\mu}$, $\lambda>-1$, $\mu\in\mathbb{C}$ and its series
representation, which converges in the whole complex $z$-plane, is
given by 
\[
M_{\beta}(z):=W_{-\beta,1-\beta}(-z)=\sum_{n=0}^{\infty}\frac{(-z)^{n}}{n!\Gamma(-\beta n+1-\beta)}.
\]
For the choices $\beta=\nicefrac{1}{2}$ and $\beta=\nicefrac{1}{3}$
the corresponding $M$-Wright functions are 
\begin{equation}
M_{\nicefrac{1}{2}}(z)=\frac{1}{\sqrt{\pi}}\exp\left(-\frac{z^{2}}{4}\right),\label{eq:MWright_Gaussian}
\end{equation}
\begin{equation}
M_{\nicefrac{1}{3}}(x)  =  3^{\nicefrac{2}{3}}\mathrm{Ai}\left(\frac{x}{3^{\nicefrac{1}{3}}}\right),\label{eq:MWright_and_Airy}
\end{equation}
where $\mathrm{Ai}$ is the Airy function see \cite{Olver2010}
for more details and properties.

The absolute moments of order $\delta>-1$ in $\mathbb{R}^{+}$ of
the density $M_{\beta}$ are finite and given by 
\begin{equation}
\int_{0}^{\infty}\tau^{\delta}M_{\beta}(\tau)\,d\tau=\frac{\Gamma(\delta+1)}{\Gamma(\beta\delta+1)}.\label{eq:moments}
\end{equation}

The grey Brownian motion has the following properties: 
\begin{enumerate}
\item For each $t\geq0$, the moments of any order are given by 
\[
\begin{cases}
\mathbb{E}(B_{\beta,\alpha}^{2n+1}(t)) & =0,\\
\noalign{\vskip4pt}\mathbb{E}(B_{\beta,\alpha}^{2n}(t)) & =\frac{(2n)!}{2^{n}\Gamma(\beta n+1)}t^{n\alpha}.
\end{cases}
\]
\item The covariance function has the form 
\begin{equation}
\mathbb{E}(B_{\beta,\alpha}(t)B_{\beta,\alpha}(s))=\frac{1}{2\Gamma(\beta+1)}\big(t^{\alpha}+s^{\alpha}-|t-s|^{\alpha}\big),\quad t,s\geq0.\label{eq:auto-cv-gBm}
\end{equation}
\item For each $t,s\geq0$, the characteristic function of the increments
is 
\begin{equation}
\mathbb{E}\big(e^{i\theta(B_{\beta}(t)-B_{\beta}(s))}\big)=E_{\beta}\left(-\frac{\theta^{2}}{2}|t-s|^{\alpha}\right),\quad\theta\in\mathbb{R}.\label{eq:cf_gBm_increments}
\end{equation}
\item The process $B_{\beta,\alpha}$ is non Gaussian, $\nicefrac{\alpha}{2}$-self-similar
with stationary increments. 
\item The $\nicefrac{\alpha}{2}$-self-similarity of the ggBm and the ergodic
theorem imply that, with probability $1$, we have 
\begin{equation}
\lim_{n\rightarrow+\infty}\sum_{i=1}^{2^{n}}\left|B_{\beta,\alpha}\left(\frac{i}{2^{n}}\right)-B_{\beta,\alpha}\left(\frac{i-1}{2^{n}}\right)\right|^{\nicefrac{2}{\alpha}}=\mathbb{E}\big(\big|B_{\beta,\alpha}(1)\big|^{\nicefrac{2}{\alpha}}\big)=:\mu_{\beta,\alpha},\label{eq:variation-ggBm}
\end{equation}
see \cite{DaSilva2018}. 
\item The ggBm is not a semimartingale. In addition, $B_{\alpha,\beta}$
cannot be of finite variation on $[0,1]$ and by scaling and stationarity
of the increment on any interval. 
\end{enumerate}
We now construct a canonical version of $B_{\beta,\alpha}$. Let $\big(C([0,1],\mathbb{R}),\mathcal{B}(C([0,1],\mathbb{R}))\big)$
be the measurable space of all real-valued continuous functions vanishing
at zero with the $\sigma$-algebra $\mathcal{B}(C([0,1],\mathbb{R}))$
generated by the cylinder sets. By using the Kolmogorov extension
theorem, the probability measure $P_{B_{\beta,\alpha}}$ induced by
$B_{\beta,\alpha}$ on $C([0,1],\mathbb{R})$ is then characterized
by the probability measure of the cylinder sets. That is, for any
$0\leq t_{1}<t_{2}<\ldots<t_{n}<\infty$, $A_{1},\ldots,A_{n}$ Borel
measurable sets and $n\geq1$ we have 
\[
P_{B_{\beta,\alpha}}\left\{ w\in C\big([0,1],\mathbb{R}\big)\middle|\,w(t_{1})\in A_{1},\ldots,w(t_{n})\in A_{n}\right\} =\int_{A_{1}\times\ldots\times A_{n}}f_{\beta}(\theta,\Sigma_{\alpha})\,d\theta_{1}\ldots d\theta_{n}.
\]
So that the coordinate process 
\[
B_{\beta,\alpha}(t)(w)=w(t),\quad w\in C([0,1],\mathbb{R}),\;t\in[0,1],
\]
is a ggBm.

\subsection{Canonical Realization of Time-Changed Brownian Motion}

In what follows we consider a filtered probability space $(\Omega,\mathcal{F},P,(\mathcal{F}_{t})_{t\ge0})$.
The filtration $(\mathcal{F}_{t})_{t\ge0}$ is assumed to satisfy
the usual conditions, that is increasing, right continuous and $\mathcal{F}_{0}$
contains all $P$-negligible events in $\mathcal{F}$ and supporting
a standard Brownian motion $W$ and an independent $\beta$-stable
process $S_{\beta}$, $0<\beta<1$.

Let $S_{\beta}=\{S_{\beta}(t),\;t\in[0,1]\}$ be a strictly increasing
$\beta$-stable subordinator, with Laplace transform given by, see
\cite[Ch.~III]{Bertoin96}
\[
\mathbb{E}(e^{-\theta S_{\beta}(t)})=e^{-t\theta^{\beta}}, \quad\theta>0,\,t\in[0,1].
\]
Define its inverse process by 
\[
U_{\beta}(t):=\inf\left\{ s\ge0\middle|S_{\beta}(t)>s\right\} ,\quad t\in\mathbb{R}^{+}.
\]

\begin{rem}
\label{rem:Inverse} 
\begin{enumerate}
\item The process $U_{\beta}$ has continuous and nondecreasing paths. In
addition, it is a $\mathbf{\beta}$-self-similar process but has neither
stationary nor independent increments, so it not a L{\'e}vy process, see
\cite{MS2004}. 
%\item \cite{Bingham1971}, showed that
%$U_{\beta}$ has a Mittag-Leffler distribution 
%\begin{equation}
%\mathbb{E}(e^{-\theta U_{\beta}(t)})=E_{\beta}(-\theta t^{\beta})\quad u>0,\,t\in\mathbb{R}^{+}.\label{eq:Laplace-U}
%\end{equation}
\item It follows from \eqref{eq:M_wright} that the distribution of $U_{\beta}(t)$,
$t>0$ is absolutely continuous with respect to the Lebesgue measure
and its density is $\mathbb{M}_{\beta}(\tau,t):=t^{-\beta}M_{\beta}(t^{-\beta}\tau)$,
$\tau\ge0$. $\mathbb{M}_{\beta}$ is called $\mathbb{M}$-Wright
function in two variables which appears as the fundamental solution
of the time fractional drift equation, see \cite{Mainardi_Mura_Pagnini_2010}
and references therein for more details. As a consequence $U_{\beta}(t)$,
$t\ge0$ is not a bounded random variable. 
\end{enumerate}
\end{rem}
We define the time-change process 
\[
U_{\beta,\alpha}(t):=U_{\beta}(f(t)),
\]
where $f(t):=t^{\nicefrac{\alpha}{\beta}},\,t\in[0,1]$. The maps
$t\mapsto U_{\beta,\alpha}(t)$, $t\in[0,1]$ are almost surely increasing
and continuous. 
\begin{lem}
The family of time-change processes $U_{\beta,\alpha}=\big\{ U_{\beta,\alpha}(t),\,t\in[0,1]\big\}$
is a nondecreasing family of $\mathcal{F}_{t}$-stopping times such
that $\mathbb{E}\left(U_{\beta,\alpha}(t)\right)=\frac{t^{\alpha}}{\Gamma(\beta+1)}$. 
\end{lem}
\begin{proof}
Indeed, for any $s\geq0$, we have 
\[
[U_{\alpha,\beta}(t)\leq s]=\left[t^{\nicefrac{\alpha}{\beta}}\leq S_{\beta}(s)\right]\in\mathcal{F}_{s}.
\]
Moreover we have from assertion 3 of Remark \ref{rem:Inverse} and
equality \eqref{eq:moments} that $\mathbb{E}\big(U_{\beta,\alpha}(t)\big)=\frac{t^{\alpha}}{\Gamma(\beta+1)}$. 
\end{proof}
We introduce the subordination of $W$ by the time-change process
$U_{\beta,\alpha}$, that is 
\[
X_{\beta,\alpha}:=\big\{ X_{\beta,\alpha}(t):=W(U_{\beta,\alpha}(t)),\,t\in[0,1]\big\}.
\]
and denote by $P_{X_{\beta,\alpha}}$ the measure induced by $X_{\beta,\alpha}$
on $\big(C([0,1],\mathbb{R}),\mathcal{B}\big(C([0,1],\mathbb{R})\big)\big)$.
The process $X_{\beta,\alpha}$ possesses the following properties. 
\begin{enumerate}
\item The time-change process $X_{\beta,\alpha}$ is a non Gaussian and
$\nicefrac{\alpha}{2}$-self-similar. In fact, the characteristic
function of $X_{\beta,\alpha}(t)=W(U_{\beta,\alpha}(t))$ is given
by 
\begin{equation}
\mathbb{E}(e^{i\theta X_{\beta,\alpha}(t)})=\int_{0}^{\infty}\mathbb{E}(e^{i\theta W(\tau)})\,dP_{U_{\beta,\alpha}(t)}(\tau)=\int_{0}^{\infty}e^{-\frac{\theta^{2}}{2}\tau}\,dP_{U_{\beta,\alpha}(t)}(\tau)=E_{\beta}\left(-\frac{\theta^{2}}{2}t^{\alpha}\right),\label{eq:one-dimensional}
\end{equation}
where in the last equality we used \eqref{eq:Laplace-U} and the fact
that $(f(t))^{\beta}=t^{\alpha}$. Since $U_{\beta}$ is $\beta$-self-similar,
then $U_{\beta,\alpha}$ is $\alpha$-self-similar. Using the fact
that $W$ is $\nicefrac{1}{2}$-self-similar, it follows that $X_{\beta,\alpha}$
is $\nicefrac{\alpha}{2}$-self-similar. 
\item The process $X_{\beta,\alpha}$ has no stationary increments, it is
not a L{\'e}vy process. Indeed, it follows from Corollary\ 2.46 in \cite{Morters_Peres}
we have 
\[
\mathbb{E}\left[\left(X_{\beta,\alpha}(t)-X_{\beta,\alpha}(s)\right)^{2}\right]  =\mathbb{E}\left[\left(X_{\beta,\alpha}(t)\right)^{2}\right]-\mathbb{E}\left[\left(X_{\beta,\alpha}(s)\right)^{2}\right].
\]
A simple computation shows that 
\[
\mathbb{E}\left[\left[\left(X_{\beta,\alpha}(t)\right)^{2}\right]\right]  =\frac{t^{\alpha}}{\Gamma(\beta+1)}\quad \mathrm{and} \quad 
\mathbb{E}\left[\left[\left(X_{\beta,\alpha}(s)\right)^{2}\right]\right]  =\frac{s^{\alpha}}{\Gamma(\beta+1)}.
\]
Therefore, we obtain 
\begin{equation}
\mathbb{E}\left[\left(X_{\beta,\alpha}(t)-X_{\beta,\alpha}(s)\right)^{2}\right]=\frac{t^{\alpha}}{\Gamma(\beta+1)}-\frac{s^{\alpha}}{\Gamma(\beta+1)}.\label{eq:1-1}
\end{equation}
\item The process $X_{\beta,\alpha}$ is an $(\mathcal{G}_{t})$-semimartinagle,
where $\mathcal{G}_{t}:=\mathcal{F}_{U_{\beta,\alpha}(t)}$, see \cite[Corollary~10.12]{Jacod1979}. 
\end{enumerate}
Let $(C(\mathbb{R}^{+},\mathbb{R}),\mathcal{B}(C(\mathbb{R}^{+},\mathbb{R})),P_{W})$
be the classical Wiener space, that is the space of all continuous
functions $w:\mathbb{R}^{+}\longrightarrow\mathbb{R}$ with $w(0)=0$,
endowed with the locally uniform convergence topology. $P_{W}$ is
the Wiener measure so that the coordinate process 
\[
W(t)(w):=w(t),\quad\forall t\ge0,\;w\in C(\mathbb{R}^{+},\mathbb{R})
\]
is a standard Brownian motion.

Let $\mathbb{U}$ be the space of all continuous nondecreasing functions
$l:[0,1]\longrightarrow\mathbb{R}^{+}$ with $l(0)=0$. The space
$\mathbb{U}$ is equipped with uniform convergence topology. Denote
by $P_{U_{\beta,\alpha}}$ the probability measure induced by $U_{\beta,\alpha}$
on $\mathbb{U}$. Then the time-change process $U_{\beta,\alpha}$
can be realized as a canonical process on $\left(\mathbb{U},\mathcal{B}\left(\mathbb{U}\right),P_{U_{\beta,\alpha}}\right)$
by 
\[
U_{\beta,\alpha}(t)(l):=l(f(t)),\,\,t\in[0,1],\,\,l\in\mathbb{U}.
\]
Moreover, the processes $S_{\beta}$ and $W$ are assumed to be mutually
independent, then we have also the independence between $W$ and the
time-change process $U_{\beta,\alpha}$.

Since $U_{\beta,\alpha}$ and $W$ are independent, $X_{\beta,\alpha}$
is the canonical process on the product space $\big(C\left(\mathbb{R}^{+},\mathbb{R}\right)\times\mathbb{U},\mathcal{B}(C\left(\mathbb{R}^{+},\mathbb{R}\right))\otimes\mathcal{B}(\mathbb{U}),P_{W}\otimes P_{U_{\beta,\alpha}}\big)$
such that 
\[
X_{\beta,\alpha}(t)(w,l):=W\big(U_{\beta,\alpha}(t)(l)\big)(w)=w\big(l(f(t))\big),\,\,t\in[0,1],\,\,w\in C\left(\mathbb{R}^{+},\mathbb{R}\right),\;l\in\mathbb{U}.
\]
Notice that we need to use the space $C(\mathbb{R}^{+},\mathbb{R})$
for the realization of $X_{\beta,\alpha}$ due to the fact that $U_{\beta}(t)$,
$t\ge0$ is not a bounded random variable, cf.\ Remark \ref{rem:Inverse}-3.
Moreover, $P_{X_{\beta,\alpha}}$ is a probability measure on the
path space 
\begin{equation}
\mathcal{X}=\left\{ w\circ l\circ f:[0,1]\longrightarrow\mathbb{R}\middle|\,w\in C\left(\mathbb{R}^{+},\mathbb{R}\right),\,\,l\in\mathbb{U}\right\} ,\label{eq:subor-full-set}
\end{equation}
equipped with the uniform convergence topology.

\section{Singularity of Generalized Grey Brownian Motion and Time-changed
Brownian Motion}

\label{sec:Singularity}In this section we establish the mutual singularity
of the probability measures $P_{B_{\beta,\alpha}}$ and $P_{X_{\beta,\alpha}}$
on $C([0,1],\mathbb{R})$, see Theorem\ \ref{thm:Singularity} below.
Let us first show that both processes $B_{\beta,\alpha}$ and $X_{\beta,\alpha}$
have only the same one dimensional distributions. 
\begin{prop}
\label{prop:1-dim-distr}The processes $B_{\beta,\alpha}$ and $X_{\beta,\alpha}$
have only the same one dimensional distribution. 
\end{prop}
\begin{proof}
For any $\theta\in\mathbb{R}$, it follows from \eqref{eq:charact-func-ggBm},
with $n=1$ and \eqref{eq:one-dimensional}, that 
\[
\mathbb{E}(e^{i\theta B_{\beta,\alpha}(t)})=\mathbb{E}(e^{i\theta X_{\beta,\alpha}(t)})=E_{\beta}\left(-\frac{\theta^{2}}{2}t^{\alpha}\right).
\]
This shows that the one dimensional distribution of $B_{\beta,\alpha}$
and $X_{\beta,\alpha}$ coincides. In order to show that the distributions
of higher order do not coincide it is sufficient to prove that, for
any $0\le s<t\le1$ 
\begin{equation}
\mathbb{E}\left[\left(B_{\beta,\alpha}(t)-B_{\beta,\alpha}(s)\right)^{2}\right]\neq\mathbb{E}\left[\left(X_{\beta,\alpha}(t)-X_{\beta,\alpha}(s)\right)^{2}\right].\label{eq:non-equal-one-dist}
\end{equation}
On one hand, it is clear that 
\[
\mathbb{E}\left[\left(B_{\beta,\alpha}(t)-B_{\beta,\alpha}(s)\right)^{2}\right]=\frac{\vert t-s\vert^{\alpha}}{\Gamma(\beta+1)}.
\]
On the order hand, from \eqref{eq:one-dimensional} we have 
\[
\mathbb{E}\left[\left(X_{\beta,\alpha}(t)-X_{\beta,\alpha}(s)\right)^{2}\right]=\frac{t^{\alpha}}{\Gamma(\beta+1)}-\frac{s^{\alpha}}{\Gamma(\beta+1)}.
\]
Therefore, the conclusion follows from the fact that for any $\alpha\in(1,2)$,
$\vert t-s\vert^{\alpha}\neq t^{\alpha}-s^{\alpha}$. 
\end{proof}
As a consequence of Proposition\ \ref{prop:1-dim-distr} the processes
$B_{\beta,\alpha}$ and $X_{\beta,\alpha}$ do not induce the same
probability measures on the path space $C([0,1],\mathbb{R})$. The
task of distinguishing between the two measures becomes a question
of interest. One possibility is to investigate the variation of the
paths under these measures.

We start by recalling the notion of $p$-variation, $p$-variation
index of a function.

Let $f$ be a function in $C([0,T],\mathbb{R})$. The $p$-variation,
$0<p<\infty$, of $f$ is defined by 
\begin{equation}
v_{p}(f;[0,T]):=\sup_{\Pi_{n}}\sum_{k=1}^{n}|f(t_{k})-f(t_{k-1})|^{p},\label{eq:-pvariation-function}
\end{equation}
where the supremum is taken over all partitions $\Pi_{n}=\{t_{0},t_{1},\ldots,t_{n}\}$,
$n\ge1$ of $[0,T]$ with $0=t_{0}<t_{1}<\ldots<t_{n}=T$. 
\begin{rem}
\label{rem:p-variation}If all sample path of a stochastic process
$X$ have bounded $p$-variation on $[0,t]$ for certain $0<p<\infty$
and $0<t\le1$, as the set of all partitions on $[0,t]$ is not measurable,
then the function $w\mapsto v_{p}(X(\cdot,w);[0,t])$ need not be
measurable. In order to overcome this difficulty, for a continuous
stochastic process, the $p$-variation over an interval is indistinguishable
from the $p$-variation over a countable and everywhere dense set,
which is always measurable. 
\end{rem}
Let $f$ be a continuous function on $[0,1]$ and $\lambda=\left\{ \lambda_{m}\middle|m\ge1\right\} $
a nested sequence of dyadic partitions $\lambda_{m}=\{i2^{-m}:i=0,\ldots,2^{m}\}$,
$m\ge1$, of $[0,1]$. For $0<p<\infty$, $m\ge1$, and $t\in[0,1]$,
let 
\[
v_{p}(f;\lambda_{m})(t):=\max\left\{ \sum_{j=1}^{k}|f(s_{i(j)}\wedge t)-f(s_{i(j-1)}\wedge t)|^{p}\middle|0=i(0)<i(1)<\ldots<i(k)=2^{m}\right\} ,
\]
where $s_{i}:=i2^{-m}$, $i=0,\ldots,2^{m}$. Since $\lambda$ is
a nested sequence of partitions, the sequence $v_{p}(f;\lambda_{m})(t)$,
$m\ge1$ is non-decreasing for each $t\in[0,1]$. For $t\in[0,1]$
we define 
\begin{equation}
v_{p}(f)(t):=\sup_{m\ge1}v_{p}(f;\lambda_{m})(t)=\lim_{m\to\infty}v_{p}(f;\lambda_{m})(t).\label{eq:limit-vp}
\end{equation}
For continuous functions $f$ the $v_{p}(f)(t)$ is equal to $v_{p}(f;[0,t])$
for any $t\in[0,1]$. For a continuous stochastic process $X=\{X(t),\;0\le t\le1\}$
and for each $t\in[0,1]$ 
\[
v_{p}(X)(t,w):=v_{p}(X(\cdot,w))(t)
\]
is possibly an unbounded but a measurable function of $w\in\Omega$.
It is clear from the definition \eqref{eq:limit-vp} that if $X$
is adapted to the filtration $(\mathcal{F}_{t})_{t\ge0}$, then $v_{p}(X)$
is also adapted to $(\mathcal{F}_{t})_{t\ge0}$ whose almost all sample
paths are non-decreasing. If $v_{p}(X)(1)<\infty$ almost surely,
then $\{v_{p}(X;[0,t]),t\in[0,1]\}$ is a stochastic process which
is indistinguishable from $\{v_{p}(X)(t),t\in[0,1]\}$, see Theorem
2, page 117 in \cite{Norvaisa2018}. 
\begin{defn}[cf.\ \cite{Norvaisa2018}]
Let $0<p<\infty$ and $X=\{X(t),\;0\le t\le1\}$ be a given stochastic
process. We say that $X$ is of bounded $p$-variation if $X$ is
adapted to $(\mathcal{F}_{t})_{t\ge0}$ and $v_{p}(X)(1)<\infty$
almost surely. We call $v_{p}(X)$ the $p$-variation process of $X$. 
\end{defn}
For a function $f:[0,1]\longrightarrow\mathbb{R}$, the $p$-variation
index of $f$ is defined by 
\[
v(f):=v(f;[0,1]):=\begin{cases}
\inf\left\{ p>0\middle|v_{p}(f;[0,1])<+\infty\right\} , & \mathrm{if\;the\;set\;is\;non\;empty},\\
+\infty, & \mathrm{otherwise}.
\end{cases}
\]
The $p$-variation index of a stochastic process $X$ is defined similarly
and is denoted by $v(X;[0,1])$. More precisely, the $p$-variation
index $v(X;[0,1])$ is defined provided $v(X(\cdot,w);[0,1])$ is
a constant for almost all $w\in\Omega$. We say that the $p$-variation
index of $X$ is $q$ if with probability $1$, the $p$-variation
index $v(X;[0,1])=q$. 
\begin{example}
\label{exa:p-variation-W} 
\begin{enumerate}
\item For a standard Brownian motion $W$ we have almost surely $v_{p}(W;[0,1])<+\infty$
for $p>2$ and $\upsilon_{2}(W;[0,1])=+\infty$. Then $\upsilon(W;[0,1])=2$,
see \cite{Levy40}.
\item If $Y$ is a semimartingale then, for any $p>2$, $Y$ has bounded
$p$-variation, see \cite{Norvaisa2018}. In addition, it is well
known that it must have unbounded $p$-variation for every $p<2$,
see \cite[Thm.~5 page~120]{Norvaisa2018}. Then we conclude that
$\upsilon(Y;[0,1])=2$. 
\item For a fractional Brownian motion $B^{H}$. We have almost surely $v_{p}(B^{H};[0,1])<+\infty$
for $p>1/H$ and $v_{1/H}(B^{H};[0,1])=+\infty$. So $\upsilon(B^{H},[0,1])=1/H$,
see \cite[Sec.~5.3]{Dudley1998}. 
\item For ggBm $B_{\beta,\alpha}$ we have almost surely $v_{p}(B_{\beta,\alpha};[0,1])<+\infty$
for $p>\nicefrac{2}{\alpha}$ and $v_{\nicefrac{2}{\alpha}}(B_{\beta,\alpha};[0,1])=+\infty$.
So $\upsilon(B_{\beta,\alpha};[0,1])=\nicefrac{2}{\alpha}$, see \cite{DaSilva2018}. 
\end{enumerate}
\end{example}
Now we consider the sets 
\begin{eqnarray}
\mathcal{V}_{\nicefrac{2}{\alpha}}&:=&\left\{ w\in C\big([0,1],\mathbb{R}\big)\middle|\upsilon(w;[0,1])=\frac{2}{\alpha}\right\} \\
\mathcal{V}_{2,loc}&:=&\left\{ w\in C\big(\mathbb{R}^{+},\mathbb{R}\big)\middle|\upsilon(w,[0,T])=2,\,\forall\,T>0\right\} .
\end{eqnarray}
Note that $\mathcal{V}_{2,loc}$ is a measurable set. In addition,
we define the subset $\hat{\mathcal{V}}$ of $\mathcal{X}$ (cf.~\eqref{eq:subor-full-set})
by 
\begin{equation}
\hat{\mathcal{V}}:=\left\{ \hat{w}:=w\circ l\circ f:[0,1]\longrightarrow\mathbb{R}\middle|\,w\in\mathcal{V}_{2,loc},\;l\in\mathbb{U}\right\} .\label{eq:subo-full-subset}
\end{equation}
It is easy to see that for any $\hat{w}\in\hat{\mathcal{V}}$, we
have $v(\hat{w};[0,1])=2.$ 
\begin{lem}
\label{lem:disjoint}The sets $\mathcal{V}_{\nicefrac{2}{\alpha}}$
and $\hat{\mathcal{V}}$ are disjoint. 
\end{lem}
\begin{proof}
Let $\hat{w}\in\mathcal{V}_{\nicefrac{2}{\alpha}}\cap\hat{\mathcal{V}}$
be given. As $\hat{w}\in\mathcal{V}_{\nicefrac{2}{\alpha}}$, it follows
from the definition of $\mathcal{V}_{\nicefrac{2}{\alpha}}$ that
\[
v(\hat{w};[0,1])=\frac{2}{\alpha}.
\]
On the other hand, $\hat{w}\in\hat{\mathcal{V}}$ then $v(\hat{w};[0,1])=2$
which is absurd because $\frac{2}{\alpha}<2$, for any $\alpha\in(1,2)$. 
\end{proof}
In order to prove the main theorem of this section we need the following
result. 
\begin{prop}[{{\cite[Thm.~1.35]{Morters_Peres}}}]
\label{prop:qv-Bm}Let $T>0$ be given and $\Pi_{n}=\{t_{0},t_{1},\ldots,t_{n}\}$,
$n\ge1$ with $0=t_{0}<t_{1}<\ldots<t_{n}=T$ be a sequence of nested
partitions of $[0,T]$, that is at each state one or more partition
points are added, such that the mesh $\|\Pi_{n}\|:=\max_{0\le i\le n}|t_{i-1}-t_{i}|\to0$,
$n\to\infty$. Then, almost surely, we have 
\begin{equation}
\lim_{n\rightarrow\infty}\sum_{i=1}^{n}\left|W(t_{i})-W(t_{i-1})\right|^{2}=T.\label{eq:qv-Bm}
\end{equation}
\end{prop}
\begin{thm}
\label{thm:Singularity}The two probability measures $P_{B_{\beta,\alpha}}$
and $P_{X_{\beta,\alpha}}$ are mutually singular. 
\end{thm}
\begin{proof}
As a consequence of Proposition~\ref{prop:qv-Bm} we have $P_{W}(\mathcal{V}_{2,loc})=1$,
then we obtain $P_{X_{\beta,\alpha}}\big(\hat{\mathcal{V}}\big)=1$.
On the other hand, \eqref{eq:variation-ggBm} implies that $P_{B_{\beta,\alpha}}\big(\mathcal{V}_{\nicefrac{2}{\alpha}}\big)=1$.
From Lemma~\ref{lem:disjoint} we have $\mathcal{V}_{\nicefrac{2}{\alpha}}\cap\hat{\mathcal{V}}=\emptyset$.
This last requirement guarantees the mutual singularity of the probability
measures $P_{B_{\beta,\alpha}}$ and $P_{X_{\beta,\alpha}}$. 
\end{proof}
\begin{rem}
The fact that $P_{X_{\beta,\alpha}}(\hat{\mathcal{V}})=1$ is also
a consequence of the semimartingale property of $X_{\beta,\alpha}$,
see Example~\ref{exa:p-variation-W}-2. 
\end{rem}

\section{Singularity of the Solutions of SDEs Driven by Generalized Grey Brownian
Motion and Time-Changed Brownian motion}

\label{sec:Applications-SDEs}In this section we show that the singularity
property from Section~\ref{sec:Singularity} propagates to the probability
measures induced by the processes which are solutions to the SDEs
driven by ggBm $B_{\beta,\alpha}$ and the time-changed Bm $W(U_{\beta,\alpha})$.
In order to show this result, at first we recall some results on the
existence and uniqueness solutions of SDEs driven by processes with
the bounded $p$-variation. In \cite{Lyons1994} considered the integral
equation 
\[
y(t)=y(0)+\int_{0}^{t}g(y(s))\,dh(s),\;0\le t\le1,
\]
where $y(0)\in\mathbb{R}$ and $h$ is continuous with bounded $p$-variation
on $[0,1]$ for some $p\in[1,2[$. He proved that this equation has
a unique solution in the space $C\mathscr{W}_{\ensuremath{p}}([0,1])$
of continuous functions of bounded $p$-variation if $g\in C^{1+\varkappa}(\mathbb{R})$
for some $\varkappa>p-1$.

Since almost all sample paths of ggBm $B_{\beta,\alpha}$ have bounded
$p$-variation for $p>\nicefrac{2}{\alpha}$, cf.\ \cite{DaSilva2018},
the SDE
\[
X(t)  =x(0)+\int_{0}^{t}g(X(s))\,dB_{\beta,\alpha}(s),\;0\le t\le1,\label{eq:SDE-ggBm-1}
\]
can be solved pathwise. It has a unique solution if $g\in C^{1+\varkappa}(\mathbb{R})$
for some $\varkappa>\nicefrac{2}{\alpha}-1$, see \cite{Kubilius2000}
and references therein. It should be noted that the solution $X$
belongs $C\mathscr{W}_{\ensuremath{p}}([0,1])$ for every $p\in]\nicefrac{2}{\alpha},\varkappa+1[$.
It follows that $\upsilon(X;[0,1])\leq\nicefrac{2}{\alpha}$. So,
we have $P_{X}\big(\mathcal{\tilde{V}}_{\nicefrac{2}{\alpha}}\big)=1$
where 
\[
\mathcal{\tilde{V}}_{\nicefrac{2}{\alpha}}:=\left\{ w\in C\big([0,1],\mathbb{R}\big)\middle|\upsilon(w;[0,1])\leq\frac{2}{\alpha}\right\} .
\]

Let us consider the stochastic differential equation driven by the
time-changed Bm 
\begin{equation}
Y(t)=y(0)+\int_{0}^{t}f(Y(s))\,dW(U_{\beta,\alpha}(s)),\quad0\le t\le1,\label{eq:SDE-time-change-Bm}
\end{equation}
where $f$ is a real-valued function, defined on $\mathbb{R}$ which
satisfies the Lipschitz condition. Then there exists a unique ($\mathcal{G}_{t}$)-semimartingale
$Y$ for which \eqref{eq:SDE-time-change-Bm} holds. Moreover this
solution can be written in the form $Y=Z\circ U_{\beta,\alpha}$,
where $Z$ satisfies the stochastic differential equation 
\[
Z(t)=y(0)+\int_{0}^{t}f(Z(s))\,dW(s),\quad0\le t\le1,
\]
see \cite[Thm.~4.2]{Kobayashi2011}. Since $Y$ is a semimartingale
then we conclude that $\upsilon(Y;[0,1])=2$, see Example~\ref{exa:p-variation-W}-2.
It follows that $P_{Y}\big(\mathcal{\tilde{\mathcal{V}}}_{2}\big)=1$
where 
\[
\mathcal{\tilde{V}}_{2}:=\left\{ w\in C\big([0,1],\mathbb{R}\big)\middle|\upsilon(w;[0,1])=2\right\} .
\]
As $1<\alpha<2$, then $\nicefrac{2}{\alpha}<2$ form which follows
that the sets $\mathcal{\tilde{V}}_{2}$ and $\tilde{\mathcal{V}}_{\nicefrac{2}{\alpha}}$
are disjoint. This shows the following theorem. 
\begin{thm}
\label{thm:Singularity-1}The two probability measures $P_{X}$ and
$P_{Y}$ are mutually singular. 
\end{thm}
\begin{rem}
We would like to mention that the result of Theorem~\ref{thm:Singularity-1}
can be extend to more general SDEs. Indeed, let us consider the following
SDE 
\[
X(t)  =x(0)+\int_{0}^{t}\sigma_{1}(X(s))\,dB_{\beta,\alpha}(s)+\int_{0}^{t}b_{1}(X(s))\,ds,\;0\le t\le1,\label{eq:SDE-ggBm-1-1}
\]
where $b_{1}$ is a Lipschitz continuous function and $\sigma_{1}\in C^{1+\varkappa}(\mathbb{R})$
for some $\varkappa>\nicefrac{2}{\alpha}-1$ with $\varkappa>p-1$.
Since almost all sample paths of ggBm $B_{\beta,\alpha}$ have bounded
$p$-variation for $p>\nicefrac{2}{\alpha}$, then this equation has
a unique solution $X$ in $C\mathscr{W}_{\ensuremath{p}}([0,1])$
for every $p\in]\nicefrac{2}{\alpha},\varkappa+1[$, see \cite{Kubilius2002}.
It follows that $\upsilon(X;[0,1])\leq\nicefrac{2}{\alpha}$.

On the other hand let us consider the following SDE 
\begin{equation}
Y(t)=y(0)+\int_{0}^{t}\sigma_{2}(Y(s))\,dW(U_{\beta,\alpha}(s))+\int_{0}^{t}b_{2}(Y(s))\,dU_{\beta,\alpha}(s),\quad0\le t\le1,\label{eq:SDE-time-change-Bm-1}
\end{equation}
where $\sigma_{2}$ and $b_{1}$ is a Lipschitz continuous functions.
Then there exists a unique ($\mathcal{G}_{t}$)-semimartingale $Y=Z\circ U_{\beta,\alpha}$
for which \eqref{eq:SDE-time-change-Bm-1} holds, where $Z$ satisfies
the stochastic differential equation 
\[
Z(t)=y(0)+\int_{0}^{t}\sigma_{2}(Z(s))\,dW(s)+\int_{0}^{t}b_{2}(Z(s))\,ds,\quad0\le t\le1,
\]
see \cite[Thm.~4.2]{Kobayashi2011}. As $Y$ is a semimartingale
then we conclude that $\upsilon(Y;[0,1])=2$. It follows that the
two probability measures are mutually singular. 
\end{rem}

\subsection*{Acknowledgments}

The first named author would like to express his gratitude to Youssef
Ouknine for the hospitality during a very pleasant stay in Mohamed
VI Polytechnic University and University Cadi Ayyad during which most
of this work was produced.

\subsection*{Founding}

This work has been partially supported by the Laboratory LIBMA from
the University Cadi Ayyad Marrakech through the program Pierre Marie
Curie ITN FP7-PEOPLE N.~213841/2008. Jos{\'e} L. da Silva is a member
of the Centro de Investiga{\c c\~a}o em Matem{\'a}tica e Aplica{\c c\~o}es
(CIMA), Universidade da Madeira, a research centre supported with
Portuguese funds by FCT (Funda{\c c\~a}o para a Ci{\^e}ncia e a
Tecnologia, Portugal) through the Project UID/MAT/04674/2013.


\begin{thebibliography}{OLBC10}

\bibitem[Ber96]{Bertoin96}
J.~Bertoin.
\newblock {\em L\'evy processes}, volume 121 of {\em Cambridge Tracts in
  Mathematics}.
\newblock Cambridge University Press, Cambridge, 1996.

\bibitem[Bin71]{Bingham1971}
N.~H. Bingham.
\newblock Limit theorems for occupation times of {M}arkov processes.
\newblock {\em Z.\ Wahrsch.\ verw.\ {G}ebiete}, 17:1--22, 1971.

\bibitem[BO09]{Beghin:2009fi}
L.~Beghin and E.~Orsingher.
\newblock Fractional {P}oisson processes and related planar random motions.
\newblock {\em Electron.\ J.\ Probab.}, 14(61):1790--1827, 2009.

\bibitem[DN98]{Dudley1998}
R.~M. Dudley and R.~Norvaisa.
\newblock {\em {An Introduction to $p$-Variation and Young Integrals}}.
\newblock Lecture notes. University of Aarhus. Centre for Mathematical Physics
  and Stochastics (MaPhySto) [MPS], 1998.

\bibitem[DSE18]{DaSilva2018}
J.~L. Da~Silva and M.~Erraoui.
\newblock Singularity of generalized grey {B}rownian motions with different
  parameters.
\newblock {\em Stoch.\ Anal.\ Appl.}, 36(4):726--732, 2018.

\bibitem[Jac79]{Jacod1979}
J.~Jacod.
\newblock {\em {Calcul Stochastique et Probl{\`e}mes de Martingales}}, volume
  714 of {\em Lecture Notes in Mathematics}.
\newblock Springer Berlin Heidelberg, Berlin, Heidelberg, 1979.

\bibitem[Kob11]{Kobayashi2011}
K.~Kobayashi.
\newblock {Stochastic calculus for a time-changed semimartingale and the
  associated stochastic differential equations}.
\newblock {\em J.\ Theor.\ Probab.}, 24(3):789--820, 2011.

\bibitem[Kub00]{Kubilius2000}
K.~Kubilius.
\newblock The existence and uniqueness of the solution of the integral equation
  driven by fractional {B}rownian motion.
\newblock {\em Lith.\ Math.\ J.}, 40:104--110, 2000.

\bibitem[Kub02]{Kubilius2002}
K.~Kubilius.
\newblock The existence and uniqueness of the solution of an integral equation
  driven by a p-semimartingale of special type.
\newblock {\em Stochastic Process.\ Appl.}, 98(2):289--315, 2002.

\bibitem[L{\'e}v40]{Levy40}
M.~P. L{\'e}vy.
\newblock Le mouvement {B}rownien plan.
\newblock {\em Amer.\ J.\ Math.}, 62(1):487--550, 1940.

\bibitem[Lyo94]{Lyons1994}
T.~Lyons.
\newblock {Differential equations driven by rough signals. I. An extension of
  an inequality of L.\ C.\ Young}.
\newblock {\em Math.\ Res.\ Lett}, 1(4):451--464, 1994.

\bibitem[MGS04]{MGS04}
F.~Mainardi, R.~Gorenflo, and E.~Scalas.
\newblock A fractional generalization of the {P}oisson processes.
\newblock {\em Vietnam J.\ Math.}, 32(Special Issue):53--64, 2004.

\bibitem[MM09]{Mura_mainardi_09}
A.~Mura and F.~Mainardi.
\newblock A class of self-similar stochastic processes with stationary
  increments to model anomalous diffusion in physics.
\newblock {\em Integr.\ Transf.\ Spec.\ F.}, 20(3-4):185--198, 2009.

\bibitem[MMP10]{Mainardi_Mura_Pagnini_2010}
F.~Mainardi, A.~Mura, and G.~Pagnini.
\newblock The {$M$}-{W}right function in time-fractional diffusion processes: A
  tutorial survey.
\newblock {\em Int.\ J.\ Differential Equ.}, 2010:Art.\ ID 104505, 29, 2010.

\bibitem[MNV11]{Meerschaert2011}
M.~M. Meerschaert, E.~Nane, and P.~Vellaisamy.
\newblock The fractional {P}oisson process and the inverse stable subordinator.
\newblock {\em Electron.\ J.\ Probab.}, 16(59):1600--1620, 2011.

\bibitem[MP08]{Mura_Pagnini_08}
A.~Mura and G.~Pagnini.
\newblock Characterizations and simulations of a class of stochastic processes
  to model anomalous diffusion.
\newblock {\em J.\ Phys.\ A: Math.\ Theor.}, 41(28):285003, 22, 2008.

\bibitem[MP10]{Morters_Peres}
P.~M{\"o}rters and Y.~Peres.
\newblock {\em {B}rownian Motion}.
\newblock Cambridge Series in Statistical and Probabilistic Mathematics.
  Cambridge University Press, Cambrige, 2010.
\newblock With an appendix by Oded Schramm and Wendelin Werner.

\bibitem[MS04]{MS2004}
M.~M. Meerschaert and H.-P. Scheffler.
\newblock {Limit theorems for continuous-time random walks with infinite mean
  waiting times}.
\newblock {\em J.\ Appl.\ Probab.}, 41(3):623--638, September 2004.

\bibitem[Nor18]{Norvaisa2018}
R.~Norvai{\v s}a.
\newblock {Quadratic Variation, $p$-Variation and Integration with Applications
  to Stock Price Modelling}, September 2018.
\newblock http://arxiv.org/abs/math/0108090v1.

\bibitem[OLBC10]{Olver2010}
F.~W.~J. Olver, D.~W. Lozier, R.~F. Boisvert, and C.~W. Clark.
\newblock {\em NIST Handbook of Mathematical Functions}.
\newblock Cambridge University Press, New York, 2010.

\bibitem[Sch90]{Schneider90}
W.~R. Schneider.
\newblock Grey noise.
\newblock In S.~Albeverio, G.~Casati, U.~Cattaneo, D.~Merlini, and R.~Moresi,
  editors, {\em {Stochastic Processes, Physics and Geometry}}, pages 676--681.
  World Scientific Publishing, Teaneck, NJ, 1990.

\end{thebibliography}
\end{document}